\documentclass[reqno, 12pt]{amsart}
\usepackage{amscd}
\usepackage{amssymb}

\usepackage{graphicx}
\usepackage{comment}
\usepackage[margin=2 in]{geometry}
\usepackage{setspace}
\usepackage{enumitem}
\usepackage{hyperref}
\usepackage[all]{xy}

\usepackage{color}

\allowdisplaybreaks[1]

\numberwithin{equation}{section}

\newcommand{\R}{\mathbb R}

\newcommand{\N}{\mathbb N}

\newcommand{\g}{\gamma}

\newcommand{\slt}{\mathfrak{sl}(2)}

\font\hb=cmbx12

\newtheorem{thm}{Theorem}[section]

\newtheorem{theorem}{Theorem}[section]
\newtheorem{definition}[theorem]{Definition}
\newtheorem{lemma}[theorem]{Lemma}
\newtheorem{remark}[theorem]{Remark}
\newtheorem{proposition}[theorem]{Proposition}
\newtheorem{corollary}[theorem]{Corollary}
\newtheorem{example}[theorem]{Example}

\newcommand{\osp}{\mathfrak{osp}(1|2)}
\newcommand{\al}{\alpha}
\newcommand{\ei}{\varepsilon}   
\newcommand{\bs}{\mathbf{e}} 
\newcommand{\coker}{\mathrm{coker}}

\newcommand{\Dlm}{\mathfrak D_{{\lambda},\mu}}
\newcommand{\Zred}[1]{Z^{#1}_{\mathrm{red}}(\mathfrak{osp}(1|2),M)}
\newcommand{\Bred}[1]{B^{#1}_{\mathrm{red}}(\mathfrak{osp}(1|2),M)}

\def\R{\mathbb R}

\def\g{\mathfrak g}

\def\R{\mathbb R}

\def\N{\mathbb N}

\def\l{\lambda}

\def\g{\mathfrak{g}}

\def\N{\mathbb{N}}

\def\R{{\mathbb R}}

\def\N{{\mathbb N}}

\newcommand{\Span}{{\rm Span}}

\newcommand{\ZD}[1]{Z^{#1}(\mathfrak{osp}(1|2),\mathfrak D_{{\lambda},\mu})}
\newcommand{\BD}[1]{B^{#1}(\mathfrak{osp}(1|2),\mathfrak D_{{\lambda},\mu})}

\newcommand{\ZM}[1]{Z^{#1}(\mathfrak{osp}(1|2),M)}
\newcommand{\BM}[1]{B^{#1}(\mathfrak{osp}(1|2),M)}

\newcommand{\SHD}[1]{\mathrm{H}^{#1}(\mathfrak{osp}(1|2),\mathfrak D_{{\lambda},\mu})}

\newcommand{\SHM}[1]{\mathrm{H}^{#1}(\mathfrak{osp}(1|2),M)}

\def\L{{\cal   L}}

\def\hb#1{\hbox{#1}}

\def\ker{\hb{ker}}

\def\hb #1{\hbox{#1}}

\def\hb#1{\hbox{#1}}

\def\dim#1{\hb{dim}(#1)}

\def\L1#1{L^1(#1)}

\def\lef({\left(}
\def\rig){\right)}

\begin{document}

\title[$\mathfrak{osp}(1|2)$-cohomology]{ $\mathfrak{osp}(1|2)$-cohomology on the space of multidifferential operators}

\begin{abstract}
We compute cohomology spaces of the
orthosymplectic Lie superalgebra $\mathfrak{osp}(1|2)$ on the
(1,1)-dimensional real superspace with coefficients in the superspace
$\frak{D}_{{\lambda},\mu}$ of $n$-ary differential operators
acting on weighted densities (${\lambda}\in\mathbb{R}^n$ and $\mu\in\mathbb{R}$).   Furthermore, we provide explicit expressions for cocycles generating these spaces.

\end{abstract}

\author[Ben Ammar]{Mabrouk Ben Ammar}
\address{Mabrouk Ben Ammar\\
 Universit\'e de Sfax, Facult\'e des Sciences, D\'epartement de Math\'ematiques\\
  Laboratoire d'Alg\`ebre, G\'eom\'etrie et Th\'eorie
 Spectrale (AGTS) LR11ES53, BP 802, 3038 Sfax, Tunisie.} \email{mabrouk.benammar@gmail.com}

 \author[Jabeur]{Amina Jabeur}
  \address{Amina Jabeur\\
 Universit\'e de Sfax, Facult\'e des Sciences, D\'epartement de Math\'ematiques\\
  Laboratoire d'Alg\`ebre, G\'eom\'etrie et Th\'eorie
 Spectrale (AGTS) LR11ES53, BP 802, 3038 Sfax, Tunisie.} \email{jabeuramina@yahoo.fr}



\date{November 10, 2026}

\maketitle
\tableofcontents
\section{Introduction}

The space of weighted densities of weight $\l$ on $\mathbb{R}$ (or $\l$-densities for short), denoted by:
\begin{equation*}
\mathcal{F}_\l=\left\{ fdx^{\l}, ~f\in C^\infty(\mathbb{R})\right\},\quad \l\in\mathbb{R},
\end{equation*}
is the space of sections of the line bundle $(T^*\mathbb{R})^{\otimes^\l}.$ The Lie algebra ${\rm Vect}(\mathbb{R})$ of vector fields $X_h=h{d\over dx}$, where $h\in C^\infty(\mathbb{R})$, acts by the {\it Lie derivative}. Alternatively, this action can be written as follows:
\begin{equation}\label{Lie1}
L_{X_h}^\l(fdx^{\l})=(hf'+\l h'f)dx^{\l},
\end{equation}
where $f'$, $h'$ are $\frac{df}{dx}$, $\frac{dh}{dx}$. \\

For ${\lambda}=(\lambda_1,\dots,\lambda_n)\in\mathbb{R}^n$ and $\mu\in\mathbb{R}$ we denote by $\mathcal{D}_{{\lambda},\mu}$ the space of $n$-ary differential operators $\mathcal{A}$ from 
$\mathcal{F}_{\l_1}\otimes\cdots\otimes\mathcal{F}_{\l_n}$ to $\mathcal{F}_\mu$. The Lie algebra ${\rm Vect}(\mathbb{R})$ acts on the space $\mathcal{D}_{{\lambda},\mu}$ of these differential operators by:
\begin{equation}\label{Lieder2}
X_h\cdot \mathcal{A}=L_{X_h}^\mu\circ\mathcal{A}-\mathcal{A}\circ L_{X_h}^{{\lambda}},
\end{equation} 
where $L_{X_h}^{{\lambda}}$ is the Lie derivative on $\mathcal{F}_{\l_1}\otimes\cdots\otimes\mathcal{F}_{\l_n}$  defined by the Leibniz rule.\\

Restricting this action to the Lie algebra $\mathfrak{sl}(2)$---which is isomorphic to the Lie subalgebra of $\mathrm{Vect}(\mathbb{R})$ spanned by $\left\{X_1,\,X_x,\,X_{x^2}\right\}$---yields 
a family of infinite-dimensional $\mathfrak{sl}(2)$-modules, still denoted by $\mathcal{D}_{\lambda,\mu}$.

Now, let us consider the superspace $\mathbb{R}^{1|1}$ endowed with its standard contact structure defined by the 1-form $\sigma$, and the Lie superalgebra $\mathcal{K}(1)$ of contact vector fields on $\mathbb{R}^{1|1}$. For $\l\in\mathbb{R}$ we introduce the $\mathcal{K}(1)$-modules $\mathfrak{F}_\lambda$ of $\lambda$-densities on
$\mathbb{R}^{1|1}$ and, for $\l=(\l_1,\dots,\l_n)\in\mathbb{R}^n$ we consider  the $\mathcal{K}(1)$-modules $\mathfrak{D}_{\lambda,\mu}$ of $n$-ary differential operators from 
$\mathfrak{F}_{\l_1}\otimes\cdots\otimes\mathfrak{F}_{\l_n}$ to $\mathfrak{F}_\mu$. These superspaces are super analogues of the spaces $\mathcal{F}_\lambda$ and
$\mathcal{D}_{\lambda,\mu}$, respectively. We  also introduce the orthosymplectic Lie superalgebra $\mathfrak{osp}(1|2)$  spanned by $\left\{X_1,\,X_{x},\,X_{x^2},\,X_\theta,\,X_{x\theta}\right\}$ 
as subalgebra of $\mathcal{K}(1)$ analogue to $\mathfrak{sl}(2)$.\\

According to Nijenhuis and Richardson \cite{r1}, the space $\mathrm{H}^1\left(\mathfrak{g};\mathrm{End}(V)\right)$ classifies the infinitesimal deformations of a 
$\mathfrak{g}$-module $V$, while the obstructions to the integrability of a given deformation belong to $\mathrm{H}^2\left(\mathfrak{g};\mathrm{End}(V)\right)$.

The computation of these cohomology spaces has a rich history in both the classical and super settings:
\begin{itemize}
  \item [1.] For $\lambda\in\mathbb{R}$, the spaces $\mathrm{H}^1_\mathrm{diff}(\mathrm{Vect}(\mathbb{R}), \mathcal{D}_{\lambda,\mu})$ were computed by Feigin and Fuchs \cite{ff}, while their 
  super analogues $\mathrm{H}^1_\mathrm{diff}\left(\mathcal{K}(1),\mathfrak{D}_{\lambda,\mu}\right)$ were determined by Basdouri  and his co-authors \cite{bbbbk}. The spaces $\mathrm{H}^1_\mathrm{diff}\left(\mathfrak{sl}(2), \mathcal{D}_{\lambda,\mu}\right)$ were computed by Gargoubi \cite{g} and Lecomte \cite{lec}, whereas the super analogues $\mathrm{H}^1_\mathrm{diff}\left(\mathfrak{osp}(1|2),\mathfrak{D}_{\lambda,\mu}\right)$ were obtained by Basdouri and Ben Ammar \cite{bab}. The $\mathfrak{sl}(2)$-relative cohomology spaces $\mathrm{H}^1_\mathrm{diff}(\mathrm{Vect}(\mathbb{R}),\mathfrak{sl}(2); \mathcal{D}_{\lambda,\mu})$ were investigated by Bouarroudj and Ovsienko \cite{bo}, and the super analogue spaces $\mathrm{H}^1_\mathrm{diff}(\mathcal{K}(1),\mathfrak{osp}(1|2); \mathfrak{D}_{\lambda,\mu})$ were established by Basdouri and his co-authors \cite{bbbbk}.
  \item [2.]  For $\lambda\in \mathbb{R}^2$, the spaces $\mathrm{H}^2_\mathrm{diff}(\mathfrak{sl}(2),\mathcal{D}_{\lambda,\mu})$ were computed by I. Basdouri et al. \cite{bhm}, and the spaces 
  $\mathrm{H}^1_\mathrm{diff}\left(\mathfrak{sl}(2),\mathcal{D}_{\lambda,\mu}\right)$ by Bouarroudj \cite{b}.
  \item [3.] For ${\lambda}\in \mathbb{R}^3$, the spaces $\mathrm{H}^1_\mathrm{diff}(\mathfrak{sl}(2),\mathcal{D}_{\lambda,\mu})$ were studied by O. Basdouri and E. Nasri \cite{be}.
  \item [4.] 
For $\lambda\in\mathbb{R}^n$ with arbitrary $n$, the spaces $\mathrm{H}^1_\mathrm{diff}(\mathfrak{sl}(2), \mathcal{D}_{\lambda,\mu})$ were studied by M. Ben Ammar and R. Sidaoui \cite{bs}, 
and the higher cohomology spaces $\mathrm{H}^j_\mathrm{diff}(\mathfrak{sl}(2), \mathcal{D}_{\lambda,\mu})$ for any $j\geq0$ were fully characterized by D. Arnal and M. Ben Ammar \cite{aba}.
\end{itemize}

In this paper, we investigate the super analogue of the work of D. Arnal and M. Ben Ammar \cite{aba}. Specifically, we study the spaces $\mathrm{H}^j_\mathrm{diff}\left(\mathfrak{osp}(1|2),\mathfrak{D}_{\lambda,\mu}\right)$ for any $j\geq0$ and $\lambda\in\mathbb{R}^n$.  Furthermore, we provide explicit expressions for cocycles generating these spaces.

\section{Definitions and Notations}
\subsection{The Lie superalgebra of contact vector fields on $\mathbb{R}^{1|1}$}

We define the superspace $\mathbb{R}^{1|1}$ in terms of its superalgebra of functions, denoted by $C^\infty(\mathbb{R}^{1|1})$ which consists of elements of the form:
\begin{equation*}
F(x,\theta)=f_0(x)+f_1(x)\theta ,
\end{equation*}
where $x$ is the even variable,  $\theta$ is  the odd variable (satisfying $\theta^2=0$), and $f_0, f_1 \in C^\infty(\mathbb{R})$. The contact bracket on $C^\infty(\mathbb{R}^{1|1})$ 
is defined by:
\begin{equation}\label{e21}
\{F,G\}=FG'-F'G+\tfrac12D(F)
\overline{D}(G),
\end{equation}
where
\begin{equation}\label{e22}
D=\partial_\theta+\theta\partial_x\quad\text{and}\quad\overline{D}=\partial_\theta-\theta\partial_x.
\end{equation}
It follows that
\begin{equation}\label{e23}
\partial_x=D^2\quad\text{and}\quad \partial_\theta=D-\theta D^2.
\end{equation}
\\

The standard contact structure on the superspace $\mathbb{R}^{1|1}$ is given by the $1$-form
\begin{equation*}
\sigma=dx+\theta d\theta.
\end{equation*}
Let $\mathrm{Vect}(\mathbb{R}^{1|1})$ be the superspace of vector fields on $\mathbb{R}^{1|1}$:
\begin{equation*}\mathrm{Vect}(\mathbb{R}^{1|1})=\left\{F_0\partial_x +  F_1\partial_\theta \mid ~F_i\in C^\infty(\mathbb{R}^{1|1})\right\}.
\end{equation*}   
The subalgebra $\mathcal{K}(1)$ of contact vector fields on $\mathbb{R}^{1|1}$ is defined as the space of vector fields preserving the distribution spanned by $\langle\overline{D}\rangle$:
$$
\mathcal{K}(1)=\big\{X\in\mathrm{Vect}(\mathbb{R}^{1|1})~|~[X,\,\overline{D}]=F_X\overline{D}\quad\hbox{for some}~F_X\in C^\infty(\mathbb{R}^{1|1})\big\}. 
$$
Every contact vector field on $\mathbb{R}^{1|1}$ can be written explicitly in the form:
\begin{equation*}
X_F=F\partial_x+\tfrac12 D(F)\overline{D},\;\text{ where }\, F\in C^\infty\left(\mathbb{R}^{1|1}\right).
\end{equation*}
The Lie bracket on $\mathcal{K}(1)$ satisfies
\begin{equation*}
[X_{F},\,X_{G}]=X_{\{F,\,G\}}.
\end{equation*}
Thus, the map  $F\mapsto X_F$ is a Lie superalgebra isomorphism between $\mathcal{K}(1)$ and $C^\infty\left(\mathbb{R}^{1|1}\right)$.

\subsection{The subalgebra $\mathfrak{osp}(1|2)$}

The Lie algebra $\frak {sl}(2)$ is realized as subalgebra of the Lie algebra  $\mathrm{Vect}(\mathbb{R})$ as follows:
\begin{equation*}
\frak {sl}(2)=\text{Span}(X_1,\,X_{x},\,X_{x^2}).
\end{equation*}
Similarly, we consider the orthosymplectic Lie superalgebra $\mathfrak{osp}(1|2)$ as the subalgebra of $\mathcal{K}(1)$ defined by:
\begin{equation*}
\mathfrak{osp}(1|2)=\text{Span}(X_1,\,X_{x},\,X_{x^2},\,X_{x\theta},\,X_{\theta}).
\end{equation*}
Explicitly
\[
 X_1=\partial_x,\quad X_x=x\partial_x+\tfrac12\theta\partial_\theta,\quad X_{x^2}=x^2\partial_x+x\theta\partial_\theta,\quad X_\theta=\tfrac12D,\quad X_{x\theta}=\tfrac12xD .
\]
The space of even elements is isomorphic to $\mathfrak{sl}(2)$, while the space of odd elements is two dimensional:
\begin{equation*}
(\mathfrak{osp}(1|2))_{\bar1}=\text{Span}(X_{x\theta},\,X_{\theta}).
\end{equation*}
The commutation and anticommutation relations involving the odd elements are
\begin{equation*}
\begin{array}{llll}
&[X_{x^2},X_\theta]~=-X_{x\theta},~~&[X_x,X_\theta]~=-\frac{1}{2}
X_\theta,~~&[X_1,X_\theta]=0,\cr
&[X_{x^2},X_{x\theta}]=0,~~&[X_x,X_{x\theta}]=\frac{1}{2}
X_{x\theta},~~&[X_1,X_{x\theta}]=X_\theta,\cr
&[X_{x\theta},X_{\theta}]~=\frac{1}{2}
X_{x},~~&[X_{\theta},X_{\theta}]~=\frac{1}{2}
X_{1},~~&[X_{x\theta},X_{x\theta}]=\frac{1}{2} X_{x^2}.
\end{array}
\end{equation*}
Consequently, the Lie superalgebra $\mathfrak{osp}(1|2)$ is generated by its odd part.

\subsection{The space of weighted densities on $\mathbb{R}^{1|1}$}

We have analogous definition of weighted densities in super setting (see \cite{ab}) with $dx$ replaced by $\sigma$. The elements of these spaces are indeed (weighted) densities since all spaces of generalized tensor fields have just single parameter relative $\mathcal{K}(1)$ --- namely the value of $X_x$ on the lowest weight vector (the one annihilated by $X_\theta$). From this perspective the volume element (heuristically, $\lq\lq dx\frac{\partial}{\partial\theta}"$) is indistinguishable from $\sigma^{\frac{1}{2}}.$ We denote by $\mathfrak{F}_{\lambda}$ the space of all  weighted densities on $\mathbb{R}^{1|1}$ of weight $\lambda\in\mathbb{R}$:
\begin{equation*}
\mathfrak{F}_\lambda=\left\{F(x,\theta)\sigma^\lambda~~|~~F(x,\theta)\in C^\infty(\mathbb{R}^{1|1})\right\}.
\end{equation*}

The subspace $\left\{X_f=f\partial_x+\frac{1}{2}f'\theta\partial\theta \;\middle|\; f\in C^\infty(\mathbb{R})\right\}$ is a subalgebra of $\mathcal{K}(1)$ isomorphic to $\mathrm{Vect}(\mathbb{R})$.  It is straightforward to see that, as $\mathrm{Vect}(\mathbb{R})$-modules, we have:
\[
\mathfrak{F}_\lambda \simeq \mathcal{F}_\lambda \oplus \Pi(\mathcal{F}_{\lambda+\frac12}),
\]
where $\Pi$ denotes the parity reversal operator.

As a vector space, $\mathfrak{F}_\lambda$ is isomorphic to $C^\infty(\mathbb{R}^{1|1})$. However, the Lie derivative of a density $G\sigma^\lambda$ along a vector field $X_F$ in $\mathcal{K}(1)$ is given by:
\begin{equation}\label{e24}
\mathfrak{L}^\lambda_{X_F}(G\,\sigma^\lambda)=(X_F(G)+\lambda F'G)\sigma^\lambda.
\end{equation}
 For $F=\theta$ and $F=x\theta$ this reads
\[
 \mathfrak{L}^\lambda_{X_\theta}=\tfrac12D,\qquad \mathfrak{L}^\lambda_{X_{x\theta}}=T^\lambda:=\tfrac12xD+\lambda\theta .
\]
Of course, for all $\lambda$, $\mathfrak{F}_\lambda$ is a $\mathcal{K}(1)$-module:
\begin{equation*}
[\mathfrak{L}^{\lambda}_{X_F},\mathfrak{L}^{\lambda}_{X_G}] =\mathfrak{L}^{\lambda}_{[{X_F},\,X_G]}.
\end{equation*}

\subsection{Differential operators on weighted densities}

According to \eqref{e23} any differential operator $\mathcal{A}$ on $\mathbb{R}^{1|1}$ is an operator on $C^\infty(\mathbb{R}^{1|1})$ which can be written in the following form:
$$
\mathcal{A}= \sum_{i=0}^\ell{h}_i (x, \theta)D^i,\quad {h}_i (x, \theta)\in C^\infty(\mathbb{R}^{1|1}).
$$
Naturally, any differential operator defines a linear mapping from $\mathfrak{F}_\lambda$ to $\mathfrak{F}_\mu$ for any $\lambda, \mu \in \mathbb{R}$ via $F\sigma^\lambda \mapsto \mathcal{A}(F)\sigma^\mu$. More generally, if $\l=(\l_1,\ldots,\l_n)$ is in $\R^n$, and $\mu$ in $\R$, we denote by $\Dlm$ the space of $n$-ary differential operators $\mathcal{A}$ from 
$\mathfrak{F}_{\l_1}\otimes\cdots\otimes\mathfrak{F}_{\l_n}$ to $\mathfrak{F}_\mu$. These spaces are $\mathcal{K}(1)$- and $\mathfrak{osp}(1|2)$-modules for the natural action:
\begin{equation}\label{d-action}
L_{X_F}\mathcal{A}=\mathfrak{L}^\mu_{X_F}\circ\mathcal{A}-(-1)^{\mathcal{A}F}\mathcal{A}\circ \mathfrak{L}^{\lambda}_{X_F},
\end{equation}
where ${\frak L}_{X_F}^{\lambda}$ is the Lie derivative on ${\frak F}_{\lambda_1}\otimes\cdots\otimes\frak{F}_{\lambda_n}$ defined by the Leibniz rule.\\

The space $\mathfrak{D}_{\lambda,\mu}$  of multi-differential operators is spanned, as a $C^\infty(\mathbb{R}^{1|1})$-module,  by the operators $D^\alpha$,  where 
$\alpha = (\alpha_1, \dots, \alpha_n) \in \mathbb{N}^n$ and  
\begin{equation}\label{d}
 D^\alpha=D_1^{\al_1}\cdots D_n^{\al_n},
\end{equation}
where $D_i^{\al_i}$ is the operator $D^{\al_i}$ acting on the $i$-th factor of the tensor product with the Koszul rule. Therefore, any differential operator $\mathcal{A}\in\Dlm$ can be written in the following form:
$$
\mathcal{A}=\sum_\alpha h_\alpha(x,\theta)D^\alpha.
$$
\section{Cohomology}
\subsection{fundamental concepts}
We shall compute the cohomology spaces of the Lie superalgebra $\frak {osp}(1|2)$
with coefficients in $\frak{D}_{{\lambda},\mu}$, where ${\lambda}\in\mathbb{R}^n$ and $\mu\in\mathbb{R}$. First, let us recall some fundamental concepts from cohomology theory~(see, e.g.,
\cite{Fu}). Let $\frak{g}$ be a Lie superalgebra acting on a superspace $V$. The space of $j$-cochains of $\frak{g}$ with values in $V$ is the $\frak{g}$-module
\begin{equation*}
C^j(\frak{g}, V ) := \mathrm{Hom}(\Lambda^j(\frak{g}),V)\quad\text{for}\quad j\geq1,
\end{equation*}
where $\Lambda^j$ is considered in the super sense, and 
$$
C^0(\frak{g}, V) =V.
$$
For $j\geq1$ the {\it coboundary operator} $\delta=\delta_j: C^j(\frak{g}, V)\rightarrow C^{j+1}(\frak{g}, V )$ is defined by 
$$
\aligned
(\delta f)(u_0,\dots,u_j)=&\sum_{i=0}^j(-1)^{\varphi_i} u_i f(u_0,\dots,\hat{u_i},\dots,u_j)\\
&+ \sum_{0\leq i<i'\leq j}(-1)^{\varphi_{i,i'}}f([u_i,u_{i}'],u_0,\dots,\hat{u_i},\dots,\hat{u_{i'}},\dots, u_j).
\endaligned
$$
where 
$$
\varphi_i={i+u_i(f+u_0+\cdots +u_{i-1})}
$$
and
$$
 \varphi_{i,i'}={i+i'+u_i(u_0+\dots+u_{i-1})+u_{i'}(u_0+\cdots+\hat{u_i}+\cdots+u_{i'-1})}.
$$
For $j=0$ the operator $\delta$ is defined on $V$ by $$(\delta v)(u_0) : =(-1)^{u_0v} u_0\cdot v,\quad u_0\in \frak{g},\quad v\in V.$$
Since $\delta^2=0$, the kernel of $\delta_j$ denoted by $Z^j(\mathfrak{g},V)$ (the space of $j$-cocycles) contains the image $B^j(\mathfrak{g},V)$ of
$\delta_{j-1}$ (the space of $j$-coboundaries).\\

The $j^{th}$ cohomology space is the quotient space
\begin{equation*}
\mathrm{H}^j
(\mathfrak{g},V):=Z^j(\mathfrak{g},V)/B^j(\mathfrak{g},V).
\end{equation*}

\subsection{Various forms of cohomology}
In this paper we aim to study the cohomology spaces $\SHD{j}$ where $\l=(\l_1,\dots,\l_n)\in\R^n$ and $\mu\in\R$, for any $j\geq0.$ Before proceeding, let us recall that more restricted notions 
of cohomology have been considered in the literature. First, in this context, authors often restrict themselves to differential cohomology (considering only differential cochains), which defines 
the differential cohomology spaces denoted by $\mathrm{H}^j_{\mathrm{diff}} (\mathfrak{osp}(1|2),\mathfrak{D}_{{\lambda},\mu})$. This notion was used for instance in \cite{bbbbk,be,b}.  Moreover, 
for technical reasons, it is sometimes useful to restrict the analysis to polynomial functions rather than $C^\infty$ functions. Specifically, we aim to work in a space where the action of 
\(H=-X_x\) is diagonalizable; this property holds for the algebra of polynomial functions, but not for  $C^\infty$ functions.
However, it was proved in \cite{aba} that all these types of cohomology spaces coincide in the classical case. 

 In this work,  we are interesting to prove that, for any $j\geq0$, the spaces $\SHD{j}$ and $\SHM{j}$ are the same, where $M$ is the submodule of $\mathfrak{D}_{{\lambda},\mu}$ of  polynomial operators. 
 \begin{definition}\label{reduced}
 We say that  a $j$-cochain $c\in C^j(\mathfrak{osp}(1|2), M)$ is reduced if $\iota_{X_1}c=0$, i.e.\ for any $u_2,\dots,u_j$ in $\mathfrak{osp}(1|2)$, 
$$
(\iota_{X_1}c)(u_2,\dots,u_j):=c(X_1,u_2,\dots,u_j)=0.
$$
\end{definition}
The submodule of $n$-differential operators with polynomial coefficients is:
$$
M=\left\{\mathcal{A}=\sum_\alpha h_\alpha(x,\theta)D^\alpha,\quad h_\alpha(x,\theta) \text{ a polynomial functions}\right\}.
$$
We can restrict ourselves to cochains with values in $M$ and introduce the corresponding cohomology.
Let $\Zred{j}$ and $\Bred{j}$ denote the spaces of reduced $j$-cocycles and $j$-coboundaries, respectively.
 \begin{thm}\label{scohomology}
For any $j\geq0$, we have:
$$
\SHD{j}=\SHM{j}.
$$
Moreover, up to a coboundary, any $j$-cocycle $c\in \mathrm{Z}^j(\mathfrak{osp}(1|2),\mathfrak{D}_{\lambda,\mu})$ is reduced.
\end{thm}
The proof of Theorem \ref{scohomology} is based on the following lemmas.
\begin{lemma}\label{cocycles}
For each $j$, we have:
$$
\aligned
\ZD{j}&=\Zred{j}+\BD{j}\\
\ZM{j}&=\Zred{j}+\BM{j}.
\endaligned
$$
\end{lemma}

\begin{proof} 
Notice that the operator $X_1\cdot$ is simply the derivation operator with respect to $x$:
$$
X_1\cdot h_\alpha(x,\theta)D^\alpha=h'_\alpha(x,\theta)D^\alpha.
$$
Consequently, this operator is surjective, and the space $M$ is given by:
$$
M=\{\mathcal{A}\in\mathfrak D_{{\lambda},\mu}~|~\exists\, k\in\mathbb{N}, (X_1)^k\cdot \mathcal{A}=0\}.
$$
In particular, its kernel is: 
$$
\ker(X_1\cdot)=\mathrm{Span}_\R\{D^\alpha,\,\theta D^\alpha\}.
$$

{$\mathbf{Case~} j=0$.}

If $j=0$, there is nothing to prove. Indeed, there are no coboundaries, and any $0$-cocycle $c \in \Dlm$  satisfies $X_1\cdot c = 0$. This implies that $c \in M$, yields:
$$
\ZD{0}=\ZM{0}=\Zred{0}.
$$
{$\mathbf{Case~} j=1.$}

If $e$ is a 1-cochain, there exists $f \in \Dlm$ such that $X_1 \cdot f = e(X_1)$. If $e$ is a cocycle, the cocycle $c = e - \delta f$ is thus reduced (i.e. $\iota_{X_1}c=0$). 
The cocycle relations applied to the pairs $(X_1,X_x)$, $(X_1,X_{x^2})$, $(X_1,X_{\theta})$, and $(X_1,X_{x\theta})$ then yield
$$
X_1 \cdot c(X_x) = 0, \quad X_1^2 \cdot c(X_{x^2}) = 0, \quad X_1 \cdot c(X_{\theta}) = 0, \quad X_1^2 \cdot c(X_{x\theta}) = 0,
$$
which implies that all values of $c$ lie in $M$. Moreover, if $e$ has values in $M$, we can choose $f \in M$, which gives:
$$
\aligned
\ZD{1}&=\Zred{1}+\BD{1},\\
\ZM{1}&=\Zred{1}+\BM{1}.
\endaligned
$$
{$\mathbf{Case~} j\geq2.$}

First, let us establish some properties. Setting $\g=\osp$, $X=X_1$, $V$ is $\Dlm$ or $M$, and $C^j(\g,V)=\mathrm{Hom}(\Lambda^j\g,V)$. 
Recall that $X$ is \emph{even} and $$X\cdot(h_\alpha(x,\theta)D^\alpha)=h_\alpha'(x,\theta)D^\alpha.$$  For $c\in C^j(\g,V)$ put
\begin{align*}
  (ac)(u_1,\dots,u_j)&=X\cdot c(u_1,\dots,u_j),\\
 (bc)(u_1,\dots,u_j)&=\sum_{i=1}^jc(u_1,\dots,[X,u_i],\dots,u_j).
\end{align*}
That is,
\[
L_X c=ac-bc.
\]
 Let $R^j\subset C^j(\g,V)$ be the space of reduced cochains. Since $X$ is even, $\iota_X^2=0$. Cartan's formula holds also for Lie superalgebras and even $X$:
\begin{equation}\label{cartan}
 L_X=\delta\,\iota_X+\iota_X\,\delta .
\end{equation}
{\bf Properties:}

\begin{enumerate}
 \item[(P1)] $a$ and $b$ commute ($a$ acts on the values, $b$ on the arguments), and they preserve $R^j$: for $c$ reduced, $$\iota_X(bc)(u_2,..)=\sum_ic(X,\dots,[X,u_i],\dots)+c([X,X],\dots)=0.$$
 \item[(P2)] $\mathrm{ad}X$ is nilpotent on $\g$: $$X_{x^2}\mapsto2X_x\mapsto2X_1\mapsto0\quad\text{and}\quad X_{x\theta}\mapsto X_\theta\mapsto0.$$ So $(\mathrm{ad}X)^3=0.$ Its extension to $\Lambda^j\g$ 
 is nilpotent, hence $b^m=0$ on $C^j(\g,V)$ for $m\geq2j+1$.
 \item[(P3)] $a$ is onto $V$, with right inverse $s$ defined by:
 $$
 s\,\Big(\sum h_\alpha D^\alpha\Big)=\sum\int_0^xh_\alpha D^\alpha.
 $$
  Then $a s=1$, and $s$ acts on values so it commutes with $b$ and preserves $R^j$ and $M$-valued cochains.
\end{enumerate}
\begin{lemma}\label{onto}
For every $k\ge0$, $L_X=a-b:R^k\to R^k$ is onto, with right inverse $$\mathcal R=\sum_{i=0}^{2k}s^{i+1}b^i.$$ If $c\in R^k$ has values in $M$, so does $\mathcal Rc$. In particular, 
for every cocycle $c\in Z^j(\g,V)$, $\iota_Xc\in R^{j-1}$ and there is $f_0\in R^{j-1}$ with $L_X f_0=\iota_Xc$ and $f_0$ is $M$-valued if $c$ is. Moreover, $c'=c\delta f_0$ is reduced.
\end{lemma}
\begin{proof}
Using $as^{i+1}=s^i$ and the commutation of $s$ with $b$,
\[
 (a-b)\sum_{i=0}^{2k}s^{i+1}b^i=\sum_{i=0}^{2k}s^{i}b^i-\sum_{i=1}^{2k+1}s^ib^i=1-s^{2k+1}b^{2k+1}=1.
\]
Finally $\iota_X\iota_Xc=0$, so $\iota_Xc\in R^{j-1}$.

Now, we have $L_X f_0=\iota_Xc$, then, if we put $c'=c-\delta f_0$, we get  $$\iota_X(\delta f_0)=L_Xf_0-\delta\iota_Xf_0=L_Xf_0,$$ hence $c'=c-\delta f_0$ is a cocycle with $\iota_Xc'=0$, i.e. $c'$ is reduced.
\end{proof}

\begin{lemma}\label{ker}
If $c\in C^j(\g,V)$ satisfies $L_X c=0$ then $c$ has values in $M$.
\end{lemma}
\begin{proof}
If $L_X c=0$ then $ac=bc$, so $a^mc=b^mc=0$, therefore, for $m=2j+1$ we deduce from (P2), that $X^m\cdot c(u_1,\dots,u_j)=(b^m c)(u_1,\dots,u_j)=0$ for all $u_i$, therefore $c(u_1,\dots,u_j)\in M$.
\end{proof}
Thus, for $j\geq2$
$$
\aligned
\ZD{j}&=\Zred{j}+\BD{j},\\
\ZM{j}&=\Zred{j}+\BM{j}.
\endaligned
$$
Indeed, by Lemma \ref{onto}, up to a couboundary, any cocycle $c$ is reduced and then Cartan's formula gives $L_Xc=\delta\iota_Xc+\iota_X\delta c=0$, and therefore, according to Lemma \ref{ker}, the cocycle $c$ is $M$-valued.

\end{proof}
\begin{lemma}\label{M}
For any $j\ge0$: $B^j(\g,M)=Z^j(\g,M)\cap B^j(\g,\Dlm)$.
\end{lemma}
\begin{proof}
We have $\BM{j}\subset \ZM{j}\cap \BD{j}$. So, let us prove the inverse inclusion. This is holding  if $j=0$ since the only coboundary is then 0.\\
{$\mathbf{Case~} j=1$}

Let $c$ be in $\ZM{1}\cap\BD{1}$. This means that there is  $f$ in $\Dlm$ such that $c=\delta f$. Then $X_1\cdot f=c(X_1)\in M$, therefore $f$ itself belongs to $M$, 
and $c$ in $\BM{1}$. This proves the lemma if $j=1$.\\
{$\mathbf{Case~} j\geq2$}

Let $c\in Z^j(\g,M)$ and $c=\delta f$ with $f\in C^{j-1}(\g,\Dlm)$ and we have 
$$
\iota_Xf\in C^{j-2}\quad\text{and}\quad \iota_X(\iota_Xf)=0.
$$
So, $\iota_Xf\in R^{j-2}$, hence by Lemma~\ref{onto} there is $f_0\in R^{j-2}$ with 
$$
L_X f_0=\iota_Xf.
$$ 
Put $f'=f-\delta f_0$. Then $\delta f'=c$ and 
$$
\iota_Xf'=\iota_Xf-L_X f_0+\delta\iota_Xf_0=0.
$$
 By \eqref{cartan}, 
 $$
 L_Xf'=\iota_X\delta f'=\iota_Xc,
 $$ 
hence, $L_Xf'$ has values in $M$. So, $a^k(L_X f')=0$ for some $k$; since $(a-b)a^kf'=0$ we get $a^{k+1}f'=ba^kf'$, hence 
$$
a^{k+i}f'=b^ia^kf'\quad\text{and}\quad a^{k+m}f'=b^ma^kf'=0\quad\text{for}\quad m=2j-1.
$$
Thus $f'$ has values in $M$, thus, $c=\delta f'\in B^j(\g,M)$.
\end{proof}
These previous Lemmas give the identities between cohomologies and then give the proof of Theorem \ref{scohomology}.

\subsection{The groups $\mathrm{H}^j(\mathfrak{osp}(1|2),M)$}

 Now, let 
 $$
 (-X_x,X_1,-X_{x^2},2X_\theta,X_{x\theta})=(H,X,Y,A,B),
 $$
 
  so that 
  \begin{equation}\label{rel}
A^2=X,\qquad B^2=-\tfrac14Y,\qquad AB+BA=-H .
\end{equation}
 Consider the $\mathfrak{osp}(1|2)$ module $M$.  Obviously the $H$ action is diagonalized on $M$, that is we decompose $M$  into weight spaces $M^r$. That is, if $\Sigma\subset\mathbb R$ is the set of weights, then we have:
$$
M=\bigoplus_{r\in\Sigma}M^r,\quad Hv_r=r v_r,~~\forall v_r\in M^r.
$$
If $V$ is a $H$-invariant vector subspace, then $V$ itself can be decomposed in $V=\bigoplus_{r \in \Sigma}V^r$ with $V^r=M^r\cap V$. A useful consequence of \eqref{rel}: for $v\in M_r$,
\begin{equation}\label{ABBA}
ABv=-rv-BAv .
\end{equation}

 The following Theorem (Theorem 3.1 from \cite{DMB}) then holds for our $\mathfrak{osp}(1|2)$-module $M$; it will be useful to us in what follows. 
\begin{theorem}\label{coho}{\rm (The groups $\mathrm{H}^j(\mathfrak{osp}(1|2),M)$)}

Let us denote by $\ker A$ (respectively $\ker B$) the subspaces of $M$, kernel of the morphism $v\mapsto Av$ (respectively $v\mapsto Bv$) in $M$. Then we have the following linear isomorphisms:
\begin{itemize}
\item[\sl(i)]  $\mathrm{H}^0(\mathfrak{osp}(1|2),M)=\ker A\cap \ker B$.
\item[\sl(ii)]  $\mathrm{H}^1(\mathfrak{osp}(1|2),M)\simeq(\ker A\cap\ker B)\oplus (\ker A)^{-\frac{1}{2}} /B((\ker A)^0)$.
\item[\sl(iii)]  $\mathrm{H}^2(\mathfrak{osp}(1|2),M)\simeq(\ker A)^{-\frac{1}{2}}/B((\ker A)^0)$.
\item[\sl(iv)] $\mathrm{H}^j(\mathfrak{osp}(1|2),M)=0$ if $j>2$.
\end{itemize}
\end{theorem}

\begin{remark}\label{normalization}
Theorem~3.6 is stated in [1] for generators $(H,X,Y,A,B')$ satisfying
$A^2=X$, $B'^2=-Y$ and $AB'+B'A=2H$. With the normalization (3.2) used here,
the two sets of generators are related by
\[
B'=-2B ,
\]
the elements $H,X,Y,A$ being unchanged. Indeed, $B'^2=4B^2=-Y$ and
$AB'+B'A=-2(AB+BA)=2H$. 
\end{remark}

\section{Actions of $X$, $H$, $A$, $B$}
Recall that the module $\mathfrak D_{\lambda,\mu}$ of multi-differential operators is spanned, as a $C^\infty(\mathbb R)$-module, by the operators 
$$
D^\alpha\quad\text{ and }\quad \theta D^\alpha,
$$ 
where $\alpha=(\alpha_1,\dots,\alpha_n)\in\mathbb N^n$, $D=\partial_\theta+\theta\partial_x$ (so $D^2=\partial_x$) and
$D^\alpha$ is defined by \eqref{d}. Throughout, $n\ge1$, $\lambda=(\lambda_1,\dots,\lambda_n)\in\R^n$, $\mu\in\R$, and
\[
 d=\mu-|\lambda|,\qquad |\lambda|=\sum_i\lambda_i,\qquad N=2d .
\]
\subsection{The action of $D^k$}
\begin{lemma}\label{lem:comm}
For every $k\ge0$ and every function $g$:
\begin{enumerate}
 \item[(a)] $D^k(xg)=xD^kg+p(k)\,\theta D^{k-1}g+\lfloor k/2\rfloor D^{k-2}g$;
 \item[(b)] $D^k(\theta g)=(-1)^k\theta D^kg+p(k)D^{k-1}g$;
 \item[(c)] $D^kT^\lambda f=\tfrac12xD^{k+1}f+\big(\tfrac12p(k)+(-1)^k\lambda\big)\theta D^kf+\big(\tfrac12\lfloor k/2\rfloor+\lambda p(k)\big)D^{k-1}f$;
 \item[(d)] $[X_x,D^k]=-\tfrac k2D^k$,
\end{enumerate}
where $p(k)$ is the parity of $k$ (=0 if $k$ is even, 1 if $k$ is odd) and terms with a negative power of $D$ are zero.
\end{lemma}
\begin{proof}
First, observe that 
$$
D(xg)=\theta g+xDg\quad\text{and}\quad D(\theta g)=g-\theta Dg.
$$

(a) By induction: if $k=2j$ then 
$$
D^{2j}(xg)=xD^{2j}g+jD^{2j-2}g.
$$
 Applying $D$ gives 
 $$
 xD^{2j+1}g+\theta D^{2j}g+jD^{2j-1}g,
 $$
which is the formula for $k=2j+1$.\\ If $k=2j+1$, applying $D$ to the right-hand side gives $$\theta D^{2j+1}g+xD^{2j+2}g+D^{2j}g-\theta D^{2j+1}g+jD^{2j}g,$$ and the $\theta$-terms cancel: 
this is the formula for $k=2j+2$. \\

(b) We have $D^2(\theta g)=Dg-D(\theta Dg)=\theta D^2g$, so 
$$
D^{2j}(\theta g)=\theta D^{2j}g\quad\text{and}\quad D^{2j+1}(\theta g)=D^{2j}g-\theta D^{2j+1}g.
$$

 (c) Apply (a) to $g=Df$ and (b) to $g=f$ in $D^kT^\lambda f=\tfrac12D^k(xDf)+\lambda D^k(\theta f)$. \\

(d) $[X_x,D]=-\tfrac12D$ (direct computation), and $X_x$ is an even derivation.
\end{proof}

According to Theorem \ref{scohomology}, we restrict ourselves to the submodule $M$ of polynomial operators, which admits the following basis as an $\mathbb R$-vector space:
\[
 a_{m,\alpha}=x^mD^\alpha,\qquad b_{m,\alpha}=x^m\theta D^\alpha .
\]

We denote $d=\mu-|\lambda|$ and $N=2d$.  We denote $\varepsilon_i=\varepsilon_i(\alpha)=(-1)^{\alpha_1+\dots+\alpha_{i-1}}$, $O(\alpha)=\{i:\alpha_i\text{ odd}\}$, $n_o(\alpha)=|O(\alpha)|$, and
\[
\kappa(k,l)=\begin{cases} l+\frac{k-1}{4}& k \text{ odd},\\ -\frac k4 & k\text{ even},\end{cases}\qquad \kappa_i=\kappa_i(\alpha)=\kappa(\alpha_i,\lambda_i).
\]
The parity of $a_{m,\alpha}$ is $|\alpha|\bmod 2$, and that of $b_{m,\alpha}$ is $|\alpha|+1\bmod 2$. 
\subsection{The action of $A$, $B$, $X$}
\begin{proposition}
\begin{align}
Ha_{m,\alpha}&=\big(\tfrac{|\alpha|}2-d-m\big)a_{m,\alpha},& Hb_{m,\alpha}&=\big(\tfrac{|\alpha|}2-d-m-\tfrac12\big)b_{m,\alpha},\label{actH}\\
Xa_{m,\alpha}&=m\,a_{m-1,\alpha},& Xb_{m,\alpha}&=m\,b_{m-1,\alpha},\label{actX}\\
Aa_{m,\alpha}&=m\,b_{m-1,\alpha}+2\!\sum_{i\in O(\alpha)}\!\varepsilon_i a_{m,\alpha+e_i},& Ab_{m,\alpha}&=a_{m,\alpha}-2\!\sum_{i\in O(\alpha)}\!\varepsilon_i b_{m,\alpha+e_i},\label{actA}
\end{align}
\begin{align}
Ba_{m,\alpha}&=\Big(\tfrac m2+d+\tfrac{n_o(\alpha)}2\Big)b_{m,\alpha}+\sum_{i\in O(\alpha)}\varepsilon_i a_{m+1,\alpha+e_i}+\sum_{\alpha_i>0}\varepsilon_i\kappa_i a_{m,\alpha-e_i},\label{actBa}\\
Bb_{m,\alpha}&=\tfrac12a_{m+1,\alpha}-\sum_{i\in O(\alpha)}\varepsilon_i b_{m+1,\alpha+e_i}-\sum_{\alpha_i>0}\varepsilon_i\kappa_i b_{m,\alpha-e_i}.\label{actBb}
\end{align}
 By \eqref{rel}, $Y=-4B^2$ acts by $-4B\circ B$.
\end{proposition}

\begin{proof} The action of $H$ is a straightforward computation, noting that—unlike in the classical case \cite{DMB}—the terms $\tfrac{|\al|}{2}$  appear here instead of $|\al|$, because the operator D is of degree $\tfrac12$.
 The action of $X$ is obvious.\\
For the action of $A$, we have $A=2X_\theta$ and $\mathfrak{L}_{X_\theta}=\tfrac12D$, hence, for any operator $\mathcal{A}\in M$ we have
 \[
 A\cdot\mathcal{A}=D\circ\mathcal{A}-(-1)^\mathcal{A}\mathcal{A}\circ D.
\]
  For $\mathcal{A}=hD^\alpha$, we get
\[
  D(hD^\alpha f)=D(h)D^\alpha f+(-1)^{h}hD(D^\alpha f),
 \]
  and the Koszul rule gives
\[
 D\circ D^\alpha=\sum_i\ei_i(\alpha)\,D^{\alpha+\bs_i},\qquad D^\alpha\circ D=\sum_i(-1)^{\alpha_{i+1}+\cdots+\alpha_n}D^{\alpha+\bs_i}.
\]
Since 
\[
(-1)^{|\alpha|}(-1)^{\alpha_{i+1}+\cdots+\alpha_n}=\ei_i(\alpha)(-1)^{\alpha_i},
\]
 the commutator is 
\[
\sum_i\ei_i(1-(-1)^{\alpha_i})D^{\alpha+\bs_i}=2\sum_{i\in O(\alpha)}\ei_iD^{\alpha+\bs_i},
\] 
which gives the action of $A$ on $a_{m,\alpha}$ and $b_{m,\alpha}$.\\

For the action of $B$ write 
\[
  B\cdot\mathcal{A}=\mathfrak{L}^\mu_{X_{x\theta}}\circ\mathcal{A}-(-1)^{\mathcal{A}}\mathcal{A}\circ \mathfrak{L}^\lambda_{X_{x\theta}}
\]
  with 
\[
 \mathfrak{L}^\mu_{X_{x\theta}}(h\,g)=\tfrac12xD(h)g+\tfrac12(-1)^{h}xhDg+\mu\theta hg,
 \]
 \[
 D(D^\alpha )=\sum_i\ei_i\,D^{\alpha+\bs_i}
 \]
 and
\begin{align*}
 hD^\al\circ \mathfrak{L}^\lambda_{X_{x\theta}}=\sum_i(-1)^{\alpha_{i+1}+\cdots+\alpha_n}h\Big[&\tfrac12x\,D^{\alpha+\bs_i}+\ei_i\big(\tfrac12p(\al_i)+
 (-1)^{\alpha_i}\lambda_i\big)\theta D^\alpha\\
 &+\big(\tfrac12\lfloor\tfrac{\alpha_i}2\rfloor+\lambda_ip(\al_i)\big)D^{\alpha-\bs_i}\Big].
\end{align*}
Now, since $-(-1)^{hD^\al}(-1)^{\alpha_{i+1}+\cdots+\alpha_n}=-(-1)^{h}\ei_i(-1)^{\alpha_i}$, we have:\\
 $(i)$ The $x\,hD^{\alpha+\bs_i}$ terms: 
 $$
 \tfrac12(-1)^{h}\ei_i\,(1-(-1)^{\alpha_i})=(-1)^{h}\ei_ip(\al_i);
 $$
$(ii)$ The $\theta$-terms: 
$$
-(-1)^{h}\sum_i\big(\lambda_i-\tfrac12p(\al_i)\big)h\theta D^\alpha+\mu\theta hD^\alpha;
$$
 using $h\theta=(-1)^{h}\theta h$ this is 
 $$
 (-1)^{h}(\mu-|\lambda|+\tfrac12n_o(\alpha))\theta hD^\alpha;
 $$
$(iii)$  The lowering terms: for $\alpha_i$ odd, 
$$
(-1)^{h}\ei_i(\lambda_i+\tfrac{\alpha_i-1}4),
$$ 
for $\alpha_i$ even, 
$$
-(-1)^{h}\ei_i\tfrac{\alpha_i}4,
$$ 
i.e.\ $(-1)^{h}\ei_i\kappa_i(\alpha)$.
\end{proof}

\subsection{The kernels of $A$ and $B$}
Now, using the formulae of the action of the operator $A$ we prove that 
\begin{proposition}[$\ker A$]\label{kerA}
$$
\ker A=\mathrm{Span}\Big\{{\omega^\alpha~|~ \alpha\in\mathbb N^n}\Big\},
$$
where 
\begin{equation}\label{e32}
\omega^\al=A b_{0,\al}=a_{0,\alpha}-2\sum_{i:\,\al_i\, \mathrm{odd}}\ei_i(\alpha)\,b_{0,\alpha+\bs_i}.
\end{equation}
\end{proposition}

\begin{proof} ~~
\begin{itemize}
  \item [(i)] First, $\omega^\al$ is in $\ker A$ since $A\omega^\al=A^2b_{0,\al}=Xb_{0,\al}=0.$
  \item [(ii)] If $v=\sum p_\al a_{0,\alpha}+\sum q_\al b_{0,\alpha}$ is in $\ker A$ then $v'=v-\sum p_\al \omega^{\alpha}=\sum q'_\al b_{0,\alpha}$ is also in $\ker A$. But the coefficient of 
  $a_{0,\al}$ in $A(v')$ is $q'_\al$, therefore $q'_\al=0$ and  $v=\sum p_\al a_{0,\alpha}.$ 
 \end{itemize}
\end{proof}
\begin{proposition}[Action of $B$ on $\omega^\alpha$]\label{Bomega}
For all $\alpha\in\N^n$,
\begin{equation}\label{Bom}
B\omega^\alpha=\Big(d-\frac{|\alpha|}2\Big)b_{0,\alpha}+\sum_{j:\alpha_j>0}\varepsilon_j(\alpha)\kappa_j(\alpha)\,\omega^{\alpha-e_j}.
\end{equation}
In particular, if $|\alpha|=N$,
\begin{equation}\label{Bom2}
B\omega^\alpha=\sum_{j:\alpha_j>0}\varepsilon_j(\alpha)\kappa_j(\alpha)\,\omega^{\alpha-e_j}.
\end{equation}
\end{proposition}

\begin{proof}
We write $O=O(\alpha)=\{i:\alpha_i\text{ odd}\}$ and $\omega^\alpha=a_{0,\alpha}-2\sum_{i\in O}\varepsilon_i(\alpha)b_{0,\alpha+e_i}$. 

\emph{Term 1.} By \eqref{actBa} ($m=0$):
\begin{equation}\label{actBa}
Ba_{m,\alpha}=\Big(\tfrac m2+d+\tfrac{n_o(\alpha)}2\Big)b_{m,\alpha}+\sum_{i\in O(\alpha)}\varepsilon_i a_{m+1,\alpha+e_i}+\sum_{\alpha_i>0}\varepsilon_i\kappa_i a_{m,\alpha-e_i}.
\end{equation}
\[
Ba_{0,\alpha}=\Big(d+\tfrac{|O|}2\Big)b_{0,\alpha}+\sum_{i\in O}\varepsilon_ia_{1,\alpha+e_i}+\sum_{j:\alpha_j>0}\varepsilon_j\kappa_j(\alpha)a_{0,\alpha-e_j}.
\]
\emph{Term 2.} For $i\in O$, set $\gamma=\alpha+e_i$; then $\gamma_i$ is even and $O(\gamma)=O\setminus\{i\}$. By \eqref{actBb},
\[
Bb_{0,\gamma}=\tfrac12a_{1,\alpha+e_i}-\sum_{l\in O\setminus\{i\}}\varepsilon_l(\gamma)b_{1,\alpha+e_i+e_l}-\sum_{l:\gamma_l>0}\varepsilon_l(\gamma)\kappa_l(\gamma)b_{0,\gamma-e_l}.
\]
We multiply by $-2\varepsilon_i(\alpha)$ and sum over $i\in O$. We group the terms as follows:

\begin{enumerate}[label=(\alph*)]
\item \emph{Terms in $a_{1,\cdot}$}: we obtain $-\sum_{i\in O}\varepsilon_ia_{1,\alpha+e_i}$, which cancels out the corresponding sum in Term~1.
\item \emph{Terms in $b_{1,\cdot}$}: we obtain $2\sum_{i\neq l\in O}\varepsilon_i(\alpha)\varepsilon_l(\alpha+e_i)\,b_{1,\alpha+e_i+e_l}$. 

\noindent Now $\varepsilon_l(\alpha+e_i)=\varepsilon_l(\alpha)(-1)^{[i<l]}$; grouping the pairs $(i,l)$ and $(l,i)$ that yield the same vector, the coefficient is $2\varepsilon_i\varepsilon_l\big((-1)^{[i<l]}+(-1)^{[l<i]}\big)=0$.
\item \emph{Terms in $b_{0,\cdot}$}: we obtain
\[
2\sum_{i\in O}\varepsilon_i(\alpha)\sum_{l:\gamma_l>0}\varepsilon_l(\gamma)\kappa_l(\gamma)\,b_{0,\alpha+e_i-e_l}.
\]
For $l=i$: $\gamma_i=\alpha_i+1$ is even, $\kappa(\gamma_i,\lambda_i)=-\frac{\alpha_i+1}4$, and $\varepsilon_i(\gamma)=\varepsilon_i(\alpha)$, whence the contribution $-\frac12\sum_{i\in O}(\alpha_i+1)\,b_{0,\alpha}$.

\noindent For $l\ne i$ (so $\alpha_l>0$, $\gamma_l=\alpha_l$): the coefficient is $$2\varepsilon_i(\alpha)\varepsilon_l(\alpha)(-1)^{[i<l]}\kappa_l(\alpha)\,b_{0,\alpha+e_i-e_l}.$$
\end{enumerate}
\emph{Comparison with $\sum_l\varepsilon_l\kappa_l(\alpha)\omega^{\alpha-e_l}$.} We have $$\omega^{\alpha-e_l}=a_{0,\alpha-e_l}-2\sum_{i\in O(\alpha-e_l)}\varepsilon_i(\alpha-e_l)b_{0,\alpha-e_l+e_i},$$ where $O(\alpha-e_l)=O\setminus\{l\}$ if $\alpha_l$ is odd and $O\cup\{l\}$ if $\alpha_l>0$ is even; moreover, $\varepsilon_i(\alpha-e_l)=\varepsilon_i(\alpha)(-1)^{[l<i]}$ for $i\ne l$. 

\noindent The terms $a_{0,\alpha-e_l}$ coincide with those in Term~1. 

\noindent The terms $b_{0,\alpha-e_l+e_i}$ with $i\neq l$ have coefficient $$-2\varepsilon_l\kappa_l\varepsilon_i(-1)^{[l<i]}=+2\varepsilon_l\varepsilon_i(-1)^{[i<l]}\kappa_l$$ (since $(-1)^{[l<i]}=-(-1)^{[i<l]}$ for $i\ne l$), which matches (c). 

\noindent What remains is, on the one hand in $\sum_l\varepsilon_l\kappa_l\omega^{\alpha-e_l}$, the terms $i=l$ (existing only if $\alpha_l$ is even and $>0$): $$-2\varepsilon_l\kappa_l\varepsilon_l\,b_{0,\alpha}=\frac{\alpha_l}2b_{0,\alpha};$$ and on the other hand in $B\omega^\alpha$, the total coefficient of $b_{0,\alpha}$:
\[
d+\tfrac{|O|}2-\tfrac12\sum_{i\in O}(\alpha_i+1)=d-\tfrac12\sum_{i\in O}\alpha_i .
\]
The difference is
\[
d-\tfrac12\sum_{i\in O}\alpha_i-\tfrac12\sum_{\alpha_l\text{ even}}\alpha_l=d-\tfrac{|\alpha|}2,
\]
which gives \eqref{Bom}.
\end{proof}

\begin{corollary}~~
 \begin{itemize}
                \item [1.] If $N\notin\mathbb{N}$ then $\ker A\cap\ker B=0.$
                \item [2.] If $N\in\mathbb{N}$ then the subspace $\ker A\cap\ker B$ is spanned by the operators of the type
\[
\psi=\sum_{|\alpha|=N}c_\alpha \omega^\alpha,
\]
where, $\omega^\al$ is given by \eqref{e32}, and for any $\beta\in \mathbb{N}^n$ with $|\beta|=N-1$, the coefficients $c_{\beta+\bs_j}$ satisfy the following equation:
\begin{equation}\label{e33}
\sum_{j=1}^n\ei_j(\beta)\,\kappa_j(\beta+\bs_j)\,c_{\beta+\bs_j}=0.
\end{equation}
              \end{itemize}
\end{corollary}
\begin{proof}
By Proposition \ref{kerA} the operators $\omega^\al$, with $\al\in\mathbb{N}^n$, generate $\ker A$. The condition $|\al|=N$  becomes from the fact that $H\psi=0$, and the equation \eqref{e33} expresses the condition $B\psi=0$.  Indeed, the operator $\omega^\al$ has weight $\frac{|\alpha|}2-d$ and, as element of  $\ker A\cap\ker B$, it is killed by $H=-(AB+BA)$, hence has weight $0$; thus it is $0$ unless $N=2d\in\mathbb N$. 
Thus, using \eqref{Bom2} we deduce the equation \eqref{e33} from the condition $B\psi=0$.
\end{proof}
\begin{corollary} If $N\in \mathbb{N}$ then
\[
 (\ker A)^{0}={\mathrm{Span}\{\omega^\al\ :\ |\al|=N\}}\quad\text{and}\quad(\ker A)^{-\frac12}={\mathrm{Span}\{\omega^\al\ :\ |\al|=N-1\}}.
\]
($(\ker A)^{-\frac12}=0$ if $N=0$).
\end{corollary}

\begin{corollary}\label{cordim} 
In the generic case ($N-1\in \mathbb{N}$ and some $\lambda_i\notin-\tfrac12\mathbb N$)
\begin{align*}
 \dim{\ker A}^{0}&=\binom{N+n-1}{n-1},\\
 \dim{\ker A}^{-\frac12}&=\binom{N+n-2}{n-1},\\
 \dim{(\ker A)^{-\frac12}/B\big((\ker A)^0\big)}&=0.
\end{align*}
\end{corollary}
\begin{proof} In fact, $\dim{\ker A}^{0}=\#\{\al=(\al_1,\,\dots,\, \al_n)\in\mathbb{N}^n~:~|\al|=N\}$ and $\dim{\ker A}^{-\frac12}=\#\{\al=(\al_1,\,\dots,\, \al_n)\in\mathbb{N}^n~:~|\al|=N-1\}$. 
Moreover, we prove that the map 
\begin{align*}
F: (\ker A)^{0}&\rightarrow (\ker A)^{-\frac12},\\
\omega^\alpha&\mapsto B\cdot\omega^\al=\sum_j\ei_j(\alpha)\kappa_j(\alpha)\,\omega^{\alpha-\bs_j},
\end{align*}
is surjective, therefore $ \dim{(\ker A)^{-\frac12}/B\big((\ker A)^0\big)}= 0.$ Indeed, let $\gamma$ with $|\gamma|=N-1$, we have
\[
 B\omega^{\gamma+\bs_i}=\ei_i\,\kappa_i(\gamma+\bs_i)\,\omega^\gamma+\sum_{l\ne i,\ \gamma_l>0}\ei_l\kappa_l(\gamma)\,\omega^{\gamma+\bs_i-\bs_l},
\]
and $1\le\gamma_i+1\le N$, so the first coefficient $\kappa(\gamma_i+1;\lambda_i)$ is non-zero. All the other vectors have $i$-th coordinate $\gamma_i+1>\gamma_i$. 
By descending induction on $\gamma_i$ (starting at 
$\gamma_i=N-1$, where the sum is empty) every $\omega^\gamma$ is in the image. Indeed, for $\gamma_i=|\gamma|=N-1$ we have $$B\omega^{\gamma+\bs_i}=\ei_i\,\kappa_i\,\omega^\gamma,$$ therefore 
$\omega^\gamma$ is in the image. Assume that any $\omega^\gamma$ with $|\gamma|=N-1$ and $\gamma_i>j$ for some $j$ is in the image. For any $\omega^\gamma$ with $|\gamma|=N-1$ and $\gamma_i=j$ we have
\[
 B\omega^{\gamma+\bs_i}=\ei_i\,\kappa_i\,\omega^\gamma+\sum_{l\ne i,\ \gamma_l>0}\ei_l\kappa_l\,\omega^{\gamma+\bs_i-\bs_l}.
\]
Therefore, $\omega^\gamma$ is in the image since $\ei_i\,\kappa_i\neq0$ and all the $\omega^{\gamma+\bs_i-\bs_l}$ with $l\neq i$ are in the image according to the recurrence hypothesis.
\end{proof}

\section{The main result}

Throughout, $n\ge1$, $\lambda=(\lambda_1,\dots,\lambda_n)\in\R^n$, $\mu\in\R$, and
\[
 d=\mu-|\lambda|,\qquad |\lambda|=\sum_i\lambda_i,\qquad N=2d .
\]
\begin{definition}\label{deltan}
If $\lambda_i\in-\tfrac12\N$ for all $i$, put $j_i=-2\lambda_i\in\N$, $J=\sum_ij_i$ and
\[
 \varphi_m=\#\Big\{\al\in\mathbb{N}^n:\ 0\le \al_i\le 2j_i,\ \sum_i\al_i=m\Big\},\,\text{i.e.}\, \sum_m\varphi_mq^m=\prod_{i=1}^n(1+q+\cdots+q^{2j_i}).
\]
Define $\Delta_N=\varphi_{N-1}-\varphi_N$ if $N\ge J+1$ and $\Delta_N=0$ otherwise. If some $\lambda_i\notin-\tfrac12\N$, put $\Delta_N=0$.
\end{definition}

\begin{theorem}[Dimensions]\label{mainthm}
Let $h^j=\dim{ H^j(\osp;\mathfrak D_{\lambda,\mu})}$.
\begin{enumerate}
 \item If $N\notin\N$, then $h^j=0$ for all $j$. If $N\in\N$ then $h^j=0$ for $j\ge3$.
 \item If $N=0$: $h^0=h^1=1$ and $h^2=0$.
 \item If $N\ge1$ (with the convention that $\binom{N-1}{-1}=0)$:
 \begin{center}
 \renewcommand{\arraystretch}{1.4}
 \begin{tabular}{c|c}
 $h^0$ & $\displaystyle\binom{N+n-2}{n-2}+\Delta_N$\\[10pt] \hline 
 $h^1$ & $\displaystyle\binom{N+n-2}{n-2}+2\Delta_N$\\[10pt] \hline
 $h^2$ & $\Delta_N$
 \end{tabular}
 \end{center}
 
\end{enumerate}
Moreover $\Delta_N\ne0$ implies: $\lambda_i=-j_i/2$ with $j_i\in\N$ for all $i$, and $\max(J,2j_1,\dots,2j_n)+1\le N\le 2J+1$. In that case
\begin{equation}\label{sumdelta}
 \sum_{N\ge1}\Delta_N=\varphi_J,\qquad \Delta_N=\mathrm{dim}_\R({A/gA})_{2J+1-N},
\end{equation}
where $A=\R[x_1,\dots,x_n]/(x_1^{2j_1+1},\dots,x_n^{2j_n+1})$ and $g=x_1+\dots+x_n$.
\end{theorem}

\begin{remark}
The same dimensions are obtained for $\slt$ by Arnal and Ben Ammar in \cite{aba} after the substitution $t_i\mapsto 2j_i$, $k\mapsto N$.
\end{remark}

Let 
\[
V_N=\Span\{\omega^\alpha:|\alpha|=N\}=(\ker A)^0, \, W_{N-1}=\Span\{\omega^\al:|\al|=N-1\}=(\ker A)^{-\frac12},
\]
 and 
 \[F=B|_{V_N}:V_N\to W_{N-1}.
 \]
If $N\ge1$ then
\[\dim {V_N}-\dim {W_{N-1}}=\binom{N+n-1}{n-1}-\binom{N+n-2}{n-1}=\binom{N+n-2}{n-2}
\]
 hence
\[
 \dim{\ker F}=\binom{N+n-2}{n-2}+\dim{\coker F} .
\]
If $N=0$ then $V_0=\R \omega^0$, $W_{-1}=0$, $\ker F=V_0$, $\coker F=0$.\\
So Theorem \ref{mainthm} is equivalent to: $\dim{\coker F}=\Delta_N$.

\subsection{The non-singular case}

\begin{lemma}\label{fonto}
Let $N\ge1$. Suppose that for some $i$, $\kappa(k,\lambda_i)\ne0$ for all $1\le k\le N$. Then $F:V_N\to W_{N-1}$ is onto, so $\coker F=0$. In particular this holds if 
$-2\lambda_i|\notin\{0,\dots,\lfloor (N-1)/2\rfloor\}$ for some $i$.
\end{lemma}
\begin{proof} See proof of Corollary \ref{cordim}.
\end{proof}

\subsection{The singular case}
  Now, following \cite{aba}, let us restrict our study to subspaces of $V_N$ and $W_{N-1}$ and the corresponding restriction $f$ of $F$. Let
  \[
  V_N^J=\mathrm{Span}\{\omega^\al~:~\alpha_i\leq 2j_i ~(i=1,\ldots,n)\}\subset V_N.
  \]
Similarly we define $W_{N-1}^J\subset W_{N-1}$. 
\begin{proposition}
\[
F(V_N^J)\subset W_{N-1}^J\quad\text{ and }\quad W_{N-1}=W_{N-1}^J+F(V_N).
\]
\end{proposition}
\begin{proof} Obviously we have $F(V_N^J)\subset W_{N-1}^J$. 
Let us denote by $W_{N-1}^{>J}$ the subspace of $W_{N-1}^J$ spanned by the $\omega^\al\in W_{N-1}^J$ with $\al_i>2j_i$ for some $i$. We have
\[
W_{N-1}=W_{N-1}^J\oplus W_{N-1}^{>J}.
\]
So, to prove that $W_{N-1}=W_{N-1}^J+F(V_N)$ it is sufficient to prove that $$W_{N-1}^{>J}\subset F(V_N).$$
Consider $\omega^{(N-1)\bs_i}$ with $N-1>2j_i$ for some $i$. We have
\[
F(\omega^{N\bs_i})=\ei_i\,\kappa_i\,\omega^{(N-1)\bs_i}, 
\]
 therefore $\omega^{(N-1)\bs_i}\in F(V^N)$ since $\kappa_i\neq0.$ We proceed by descending induction on $\gamma_i$. Assume that any $\omega^\gamma$ with $|\gamma|=N-1$ and $\gamma_i>j>2j_i$, for 
 some $j$, is in $F(V^N)$. For any $\omega^\gamma$ with $|\gamma|=N-1$ and $\gamma_i=j$ we have
\[
 F(\omega^{\gamma+\bs_i})=\ei_i\,\kappa_i\,\omega^\gamma+\sum_{l\ne i,\ \gamma_l>0}\ei_l\kappa_l\,\omega^{\gamma+\bs_i-\bs_l}.
\]
Therefore, $\omega^\gamma$ is in $F(V^N)$ since $\ei_i\,\kappa_i\neq0$ and all the $\omega^{\gamma+\bs_i-\bs_l}$ with $l\neq i$ are in $F(V^N)$ according to the recurrence hypothesis.
\end{proof}
\begin{corollary}
Denote by $f:V_N^J\rightarrow W_{N-1}^J$ the restriction of the map $F$. Then:
$$
\coker(f)\simeq \coker(F).
$$
\end{corollary}
\begin{proof} Indeed, it is also easy to prove that $F(V_N^{>J})\subset W_{N-1}^{>J}$, therefore 
$$
F(V_N^{>J})\cap F(V_N^{J})=0\quad\text{and}\quad F(V_N)\cap W_{N-1}^J=F(V_N^J).
$$
Thus,
\[
\coker (F)=\coker(f).
\]

\end{proof}

Now, we change the basis in $V_N^J$ and $W_{N-1}^J$, by putting:
\begin{equation}\label{Eformula}
\omega_\alpha=\prod_{j=1,\,\al_j>0}^n\frac{1}{s_j(\alpha) } \omega^\alpha,
\end{equation}
where $s_j(\al)$ is defined by the relation $s_j(\al)= s_j({\al-\bs_j})\ei_j(\alpha)\kappa_j(\alpha)$. 
Thus
$$
f(\omega_\alpha)=\sum_{j:~\alpha_j>0} \omega_{\alpha-\bs_j}.
$$

\begin{proposition}
The map $f$ is of maximal rank.
\end{proposition}
\begin{proof}
Put on the ring $\mathcal R=\R[x_1,\ldots,x_n]/(x_1^{2j_1+1},\ldots x_n^{2j_n+1})$ the basis
$$
\{e_\alpha=\prod_ix_i^{2j_i-\alpha_i}:~0\leq \alpha_i\leq 2j_i\},
$$
then, on this basis, the matrix of the multiplication by $x_1+\ldots+x_n$ acting on $\mathcal R[x_i]^{2J-N}$ with values  in$\mathcal R[x_i]^{2J-N+1}$ coincides with the matrix of $f$ on the basis $(\omega_\alpha)$. Especially, 
the maximal rank of $f$ is a consequence of the Stanley-Watanabe Lefschetz Theorem \cite{hmnw,s,w}. 
\end{proof}

\begin{corollary}
\[
\mathrm{dim}(\coker f) = \Delta_N=\max(\varphi_{N-1}-\varphi_N,0).
\] 
\end{corollary}
Theorem \ref{mainthm}, being equivalent to $\dim{\coker F}=\Delta_N$, is then proved.

\section{Cocycles generating the cohomology spaces}

\subsection{generators of $H^0$}
 $H^0=H^0(\osp,M)=\ker A\cap\ker B$. It has been proven that if $N\notin\N$, this space is zero. If $N\in\N$, it is equal to $\ker F\subset V_N$: its elements are
\[
\psi=\sum_{|\alpha|=N}c_\alpha\,\omega^\alpha\qquad\text{with}\qquad \sum_{j=1}^n\varepsilon_j(\beta)\,\kappa_j(\beta+e_j)\,c_{\beta+e_j}=0\quad\text{for all }\beta,\ |\beta|=N-1.
\]

\subsection{generators of $H^1$}~~
Following \cite{DMB} we can construct a basis for $H^1$ derived from the fact that $H^1$ is isomorphic to $(\ker A\cap\ker B)\oplus (\ker A)^{-\frac{1}{2}} /B((\ker A)^0)$.
\begin{theorem}\label{genH1}
Assume $N\in\mathbb N$ (otherwise $H^1=0$). Put $\varepsilon_N=(-1)^N$.

\smallskip\noindent
\textup{(i)} \emph{Cocycles associated with $\ker A\cap\ker B$.}
Let $\psi=\sum_{|\alpha|=N}c_\alpha\,\omega^\alpha\in\ker A\cap\ker B$ and
$\beta_\psi=\sum_\alpha c_\alpha\, b_{0,\alpha}$ (so that $A\beta_\psi=\psi$).
The cochain $c_\psi$ of parity $N \bmod 2$ defined by
\[
c_\psi(A)=0,\quad c_\psi(X)=0,\quad c_\psi(H)=\psi,\quad
c_\psi(B)=-\varepsilon_N\,\beta_\psi,\quad c_\psi(Y)=4\,B\beta_\psi
\]
is a $1$-cocycle of weight $0$, where
\[
B\beta_\psi=\sum_{|\alpha|=N}c_\alpha\Big(\tfrac12 a_{1,\alpha}
-\sum_{i\in O(\alpha)}\varepsilon_i\, b_{1,\alpha+e_i}
-\sum_{\alpha_i>0}\varepsilon_i\kappa_i\, b_{0,\alpha-e_i}\Big).
\]

\smallskip\noindent
\textup{(ii)} \emph{Cocycles associated with $\operatorname{coker}F$.}
For $w\in(\ker A)^{-1/2}=W_{N-1}$ (hence $N\ge1$), the cochain $c_w$ of
parity $N \bmod 2$ defined by
\[
c_w(B)=w,\quad c_w(Y)=-4\,\varepsilon_N\,Bw,\quad c_w(A)=c_w(X)=c_w(H)=0
\]
is a $1$-cocycle of weight $0$. It is a coboundary if and only if
$w\in B\big((\ker A)^0\big)=F(V_N)$.

\smallskip\noindent
\textup{(iii)} Let $\psi_1,\dots,\psi_p$ be a basis of $\ker F=\ker A\cap\ker B$
and $w_1,\dots,w_q$ representatives of a basis of
$\operatorname{coker}F=W_{N-1}/F(V_N)$. Then the classes of
$$c_{\psi_1},\dots,c_{\psi_p},c_{w_1},\dots,c_{w_q}$$ form a basis of
$H^1(\mathfrak{osp}(1|2),M)\simeq H^1(\mathfrak{osp}(1|2),D_{\lambda,\mu})$.

\end{theorem}

\begin{proof}
\emph{Step 1: cocycle relations.}
Let $c$ be a $1$-cochain of parity $\epsilon$ with $c(A)=c(X)=0$. The cocycle
condition reads, for homogeneous $u,v$,
\[
c([u,v])=(-1)^{u\epsilon}\,u\cdot c(v)-(-1)^{v(u+\epsilon)}\,v\cdot c(u).
\]
Using the brackets recalled above, it gives on the pairs
$(A,H),(A,B),(H,B),(B,B)$:
\begin{align}
&A\,c(H)=0, \tag{R1}\\
&A\,c(B)=-(-1)^{\epsilon}\,c(H), \tag{R2}\\
&\big(H+\tfrac12\big)c(B)=(-1)^{\epsilon}\,B\,c(H), \tag{R3}\\
&c(Y)=-4(-1)^{\epsilon}\,B\,c(B). \tag{R4}
\end{align}
Conversely, if (R1)--(R4) hold and moreover $Bc(H)=0$, $c(H)$ has weight $0$
and $c(B)$ has weight $-\frac12$, the remaining relations
($(H,Y)$, $(X,B)$, $(X,Y)$, $(Y,A)$, $(Y,B)$, \dots) follow. For instance,
using $AB=-H-BA$ one gets $ABc(B)=\frac12 c(B)$, hence
$Ac(Y)=-2(-1)^{\epsilon}c(B)$, which is the $(Y,A)$ relation, and
$Xc(Y)=A\big(Ac(Y)\big)=-2(-1)^{\epsilon}A\,c(B)=2c(H)$, which is the $(X,Y)$ relation.

\emph{Step 2: (i) and (ii).}
Take $c(H)=\psi\in\ker A\cap\ker B$ of weight $0$; then (R1) holds,
(R3) gives $\big(H+\frac12\big)c(B)=0$ (since $B\psi=0$), and (R2) requires
$Ac(B)=-(-1)^{N}\psi$. As $A\beta_\psi=\psi$ (Proposition~4.1), the choice
$c(B)=-(-1)^N\beta_\psi$ works, and (R4) gives $c(Y)=4B\beta_\psi$.
The weight of $\beta_\psi$ is $-\frac12$ because $|\alpha|=N=2d$.
For $\psi=0$ and $c(B)=w\in(\ker A)^{-1/2}$, (R2), (R3) hold trivially and
(R4) gives $c(Y)=-4(-1)^N Bw$. In both cases the remaining relations are
checked as above.

\emph{Step 3: triviality.}
If $c_w=\delta g$ for some $g$, then $c_w(A)=0$ forces $Ag=0$. Since $\delta$
preserves weights, we may replace $g$ by its weight-$0$ component $g_0$:
then $Hg_0=0$, so $g_0\in(\ker A)^0$, and $w=c_w(B)=(-1)^{g_0}Bg_0
\in B\big((\ker A)^0\big)$. Conversely, if $w=(-1)^{g}Bg$ with
$g\in(\ker A)^0$, then $\delta g$ and $c_w$ are reduced cocycles with
$\delta g(H)=Hg=0=c_w(H)$ and $\delta g(B)=c_w(B)$; by (R4) they coincide.

\emph{Step 4: basis.}
By Theorem~3.6 (transported by Remark~\ref{normalization}), every class
has a reduced representative $c$ with $c(H)=\psi\in\ker A\cap\ker B$; since
$L_H=\delta\iota_H+\iota_H\delta$, only weight-$0$ classes survive. Then
$c-c_\psi$ is a reduced cocycle with $(c-c_\psi)(H)=0$, and by (R2)
$A(c-c_\psi)(B)=0$, i.e.\ $c-c_\psi=c_w$ with $w\in(\ker A)^{-1/2}$ by (R4).
The class of $c$ is thus $[c_\psi]+[c_w]$, with $w$ well defined modulo
$B\big((\ker A)^0\big)$ by Step~3. The sum is direct: if
$c_\psi+c_w=\delta g$, evaluating at $H$ gives $\psi=\pm Hg_0=0$, and then
Step~3 applies. Hence $\dim H^1=\dim{\ker F}+\dim{\operatorname{coker}F}$.

\end{proof}

\subsection{Generators of $H^2$}
\begin{theorem}
Assume $N\in\N$, $N\ge1$. For $w\in(\ker A)^{-1/2}=W_{N-1}$, the weight-$0$ $2$-cochain $c^{(2)}_w$ of parity $N\bmod2$ defined by
\[
c^{(2)}_w(H,B)=w,\qquad c^{(2)}_w(Y,A)=-4w,\qquad
c^{(2)}_w(H,Y)=-4(-1)^N Bw,
\]
and zero on all other basis pairs, is a $2$-cocycle. It is a coboundary if and only if $w\in B((\ker A)^0)=F(V_N)$. If $w_1,\dots,w_q$ represent a basis of $\coker F=W_{N-1}/F(V_N)$, then the classes of $c^{(2)}_{w_1},\dots,c^{(2)}_{w_q}$ form a basis of $H^2(\osp,M)\simeq H^2(\osp,\Dlm)$, with $\dim{ H^2}=\dim{\coker F}$. Moreover,
\[
\iota_Hc^{(2)}_w=c_w .
\]
\end{theorem}
\begin{proof} In fact, the isomorphism here between $H^2(\osp,M)$ and $\coker F$ is none other than the one constructed in \cite{DMB}, which consists of associating with each element $\omega$ of a supplementary $\mathcal{C}$ of $B((\ker A)^0)$ in $W_{N-1}$ the 2-cocycle $c_\omega^{(2)}$ defined above.

\end{proof}
\section{Explicit examples}
We study some examples for some small values of $n$ and $N\in\mathbb{N}$. Recall the convention $\binom{N-1}{-1}=0$.

\subsection{examples of dimensionalities} 
\begin{example}[$n=1$]~~
\begin{itemize}
  \item [(a)] For $\lambda=-\frac j2$ the chain $[0,2j]$ gives $\varphi_m=1$ ($0\le m\le2j$), so $\Delta_N=1$ iff $N=2j+1$, i.e.\ $\mu-\lambda=j+\tfrac12$. In that case 
$(h^0,h^1,h^2)=(1,2,1)$.\\
  \item [(b)] For $j=0$: $\lambda=0$, $\mu=\frac12$ and the invariant operator is $$c_1=a_{0,1}-2b_{0,2}=D-2\theta\partial_x=\overline{D}.$$ For $j\ge1$ the corresponding invariant is 
the Bol-type operator of order $2j+1$, $\mathfrak{F}_{-j/2}\to \mathfrak{F}_{(j+1)/2}$. \\
  \item [(c)] For $N\ge1$, $N\ne2j+1$, all of $H^0,H^1,H^2$ vanish.\\
  \item [(d)] For $N=0$, ($\mu=\lambda$), one has $(h^0,h^1,h^2)=(1,1,0)$).

\end{itemize}~~\\
 We refined here the result of \cite{DMB} (Proposition 4.1 from \cite{DMB}).
\end{example}~~\\

\begin{example}[$n=2$, $\lambda=(-\frac{j_1}2,-\frac{j_2}2)$]~~
\begin{itemize}
  \item [(a)] $\Delta_N=1$ for $2\max(j_1,j_2)+1\le N\le2(j_1+j_2)+1$ and $0$ otherwise. Hence $(h^0,h^1,h^2)=(2,3,1)$ if $\Delta_N=1$, and $(1,1,0)$ if $\Delta_N=0$.\\
  \item [(b)] For generic $\lambda$: $(h^0,h^1,h^2)=(1,1,0)$.
  \end{itemize}

\end{example}~~\\

\begin{example}[$n=3$] ~~
\begin{itemize}
  \item [(a)] In the generic case: $h^0=h^1=N+1$, $h^2=0$.\\
  \item [(b)] For $\lambda=(-\frac12,-\frac12,-\frac12)$ ($j=(1,1,1)$, $J=3$): $(1+q+q^2)^3$ has coefficients $1,3,6,7,6,3,1$, so $\Delta_N=1,3,2,1$ for $N=4,5,6,7$.\\
  \item [(c)] For $j=(2,1,1)$ ($J=4$): $(1+\dots+q^4)(1+q+q^2)^2$ has coefficients $1,3,6,8,9,8,6,3,1$, so $\Delta_N=1,2,3,2,1$ for $N=5,\dots,9$.
\end{itemize}

\end{example}~~\\

\begin{example}[Extreme cases within the general singular situation]~~

Assume $\lambda_i=-\frac{j_i}2$ for all $i$ and let $r=\#\{i:j_i\ge1\}$. Then $\Delta_{2J+1}=1$, $\Delta_{2J}=r-1$ (if $r\ge1$), and $\Delta_{2J-1}=\binom r2$ (if $J\ge2$). 
Indeed $\varphi_{2J}=1$, $\varphi_{2J-1}=r$, $\varphi_{2J-2}=\binom r2+r$.
\end{example}
\subsection{Examples of generators}
\begin{example}[$n=1$, $\lambda=-j/2$, $N=2j+1$]
Here $d=\mu-\lambda=j+\frac12$. We have $V_N=\R\omega^{N}$, $W_{N-1}=\R\omega^{2j}$, and $\kappa(N,\lambda)=-\frac j2+\frac{N-1}4=0$, so $F=0$ (Proposition~\ref{Bomega}). Hence $\dim{\ker F}=\dim{\coker F}=1$ and $$(h^0,h^1,h^2)=(1,2,1).$$ Generators:
\begin{align*}
H^1:\ & c_\psi:\ H\mapsto\psi=\omega^N,\ B\mapsto b_{0,2j+1},\ Y\mapsto 4\,Bb_{0,2j+1},\ A,X\mapsto0;\\
& c_w\ (w=\omega^{2j}=a_{0,2j}):\ B\mapsto a_{0,2j},\ Y\mapsto 4\,Ba_{0,2j},\ A,X,H\mapsto0;\\
H^2:\ & c^{(2)}_w:\ (H,B)\mapsto a_{0,2j},\ (Y,A)\mapsto-4a_{0,2j},\ (H,Y)\mapsto 4\,Ba_{0,2j}.
\end{align*}
By \eqref{actBb} and \eqref{actBa}: $$Bb_{0,2j+1}=\frac12a_{1,2j+1}-b_{1,2j+2}-\kappa(2j+1,\lambda)\,b_{0,2j}=\frac12a_{1,2j+1}-b_{1,2j+2}$$ and $$Ba_{0,2j}=d\,b_{0,2j}+\kappa(2j,\lambda)a_{0,2j-1}=(j+\tfrac12)b_{0,2j}-\frac j2a_{0,2j-1}.$$

For $j=0$ ($\lambda=0$, $\mu=\frac12$): $$\psi=D-2\theta\partial_x=\overline{D},$$ and $c_w(B)=\mathrm{Id}$, $c_w(Y)=4Ba_{0,0}=2\theta$.
\end{example}~\\

\begin{example}[$n=2$, $\lambda=(-\frac12,-\frac12)$]
Here $j=(1,1)$, $J=2$, and $\Delta_N=1$ for $3\le N\le5$ and $0$ otherwise. Solving $F\psi=0$ and computing $\coker F$, we obtain, in the basis $\omega^{(\alpha_1,\alpha_2)}$:
\begin{center}
\begin{tabular}{c|l|l|l}
$N$ & $\ker F$ (generators of $H^0$) $\psi_1,\, \psi_2,\cdots$ & $w$ ($\coker F$) & $(h^0,h^1,h^2)$\\ \hline
$1,2$ & $\dim{\ker F}=1$ & $0$ & $(1,1,0)$\\
$3$ & $\psi_1=\omega^{(0,3)},\ \psi_2=\omega^{(3,0)}$ & $\omega^{(0,2)}$ & $(2,3,1)$\\
$4$ & $\omega^{(1,3)}-\tfrac12\omega^{(0,4)},\ \ 2\omega^{(3,1)}+\omega^{(4,0)}$ & $\omega^{(1,2)}$ & $(2,3,1)$\\
$5$ & $\omega^{(2,3)}+\tfrac12\omega^{(1,4)}+\tfrac12\omega^{(0,5)},\ \ 2\omega^{(3,2)}+\omega^{(4,1)}+\omega^{(5,0)}$ & $\omega^{(2,2)}$ & $(2,3,1)$\\
$\ge6$ & $\dim{\ker F}=\binom{N}{0}=1$ ($\Delta_N=0$) & $0$ & $(1,1,0)$
\end{tabular}
\end{center}
 For each row with $N\in\{3,4,5\}$, the classes in $H^1$ are $c_{\psi_1},c_{\psi_2},c_{w}$ and the class in $H^2$ is $c^{(2)}_w$. 
\end{example}



\begin{thebibliography}{99}
\bibitem{DMB}
Arnal D., Ben Ammar M. and Dali B.,{\it The spaces $H^n(\mathfrak{osp}(1|2);M)$ for some modules $M$}, J. Math. Phys. 51 (2010), no. 9, 093517, 15 pp.\\

\bibitem{aba} D. Arnal and M. Ben Ammar, \emph{$\mathfrak{sl}(2)$-cohomology on multidifferential
operators space}, Monatshefte für Mathematik (2026). https://doi.org/10.1007/s00605-026-02256-5\\

\bibitem{ab}
Agrebaoui B and Ben Fraj N, On the cohomology of the Lie
superalgebra of contact vector fields on $S^{1|1}$, {\it Bull.
Soc. Roy. Sci. Li\`ege} {\bf 72}, 6, 2004, 365--375.\\


\bibitem{bab}
Basdouri I, Ben Ammar M, Cohomology of $\mathfrak {osp}(1|2)$
acting on linear differential operators on the supercircle
$S^{1|1}$. Letters in Mathematical Physics (2007) 81:239--251.\\

 
 \bibitem{bbbbk}
Basdouri I, Ben Ammar M, Ben Fraj N, Boujelbene M, Kammoun
K, {\it Cohomology of the Lie superalgebra of contact vector fields on
$\mathbb{R}^{1|1} $ and deformations of the superspace of symbols},
 JNMP vol. 16, No. 4 (2009)1-37.\\
 
 \bibitem{bhm} 
I. Basdouri, S. Hammami and O. Messaoud, Second cohomology space of sl(2) acting on the space of
bilinear bidifferential operators, Revista de la Unión Matemática Argentina {\bf61} (2020), no. 1, 131–143.\\

 \bibitem{be}
Basdouri O. and Nasri E., {\it $\mathfrak{sl}(2)$-invariant 3-ary differential operators and cohomology of 
$\mathfrak{sl}(2)$ acting on 3-ary differential operators}, Int. J. Geom. Methods Mod. Phys. 14, 1750015 (2017)
 [17 pages], https://doi.org/10.1142/S0219887817500153\\


\bibitem{bs}
Ben Ammar M. and Sidaoui R., {\it Cohomology of $\frak {sl}(2)$ acting on the space of $n$-ary differential operators on $\mathbb{R}$}, Indian J. Pure Appl. Math. 52, 1263-1275 (2021). https://doi.org/10.1007/s13226-021-00012-z\\

\bibitem{b}
Bouarroudj S, {\it Cohomology of the vector fields Lie algebras
on $\mathbb{R}\mathbb{P}^1$ acting on bilinear differential
operators}, International Journal of Geometric Methods in Modern Physics
 {\bf 2}; N 1, (2005), 23-40.\\
 
 \bibitem{bo}
Bouarroudj S. and Ovsienko V., {\it Three cocycles on Diff($S^1$) generalizing the Schwarzian derivative},  Internat. Math. Res. Notices {\bf 1} (1998), 25-39.\\


\bibitem{ff}
Feigin B. L. and Fuchs D. B., {\it  Homology of the Lie algebras of vector fields on the line},  Func. Anal. Appl., 14 (1980), 201--212.\\

\bibitem{Fu}
Fuchs  D. B., {\it Cohomology of infinite-dimensional Lie algebras}, Plenum Publ. New York, 1986.\\

\bibitem{g}
Gargoubi H, {\it Sur la g\'eom\'etrie de l'espace des op\'erateurs
diff\'erentiels lin\'eaires sur $\mathbb{R}$}, Bull. Soc. Roy. Sci.
Li\`ege. Vol. 69, 1, 2000, 21–47.\\

\bibitem{hmnw}
 Harima T.,  Migliore J. C., Nagel U. and Watanabe J.,{\it The weak and strong Lefschetz properties for Artinian K-algebras}, Journal of Algebra {\bf262} (2003), no. 1, 99--126.\\ 


\bibitem{lec}
Lecomte B P A, {\it On the cohomology of
$\frak{sl}(n + 1;\mathbb{R})$ acting on differential operators and
$\frak{sl}(n + 1;\mathbb{R})$-equivariant symbols,} Indag. Math. NS.
11 (1), (2000), 95-114.\\

\bibitem{r1}
A. Nijenhuis,  R.W. Richardson, {\it Cohomology and deformations
in graded Lie algebras}. Bull. Amer. Math. Soc., 72 (1966), 1–29.
\bibitem{s}
Stanley R. P.,{\it  Weyl groups, the hard Lefschetz theorem, and the Sperner property}, SIAM Journal on Algebraic and Discrete Methods {\bf1} (1980), no. 2, 168--184.\\

\bibitem{w}
Watanabe J.,{\it The Dilworth number of Artinian rings and finite posets with rank function}, in Commutative Algebra and Combinatorics (Kyoto, 1985), Advanced Studies in Pure Mathematics {\bf11} (1987), 303--312.

\end{thebibliography}
\end{document}